\documentclass[11pt,reqno]{amsart}
\usepackage{}
\usepackage{a4wide}
\numberwithin{equation}{section}
\usepackage{mathrsfs}
\usepackage{amsfonts}
\usepackage{amsmath}
\usepackage{stmaryrd}
\usepackage{amssymb}
\usepackage{amsthm}
\usepackage{mathrsfs}
\usepackage{url}
\usepackage{amsfonts}
\usepackage{amscd}
\usepackage{indentfirst}
\usepackage{enumerate}
\usepackage{amsmath,amsfonts,amssymb,amsthm}
\usepackage{amsmath,amssymb,amsthm,amscd}
\usepackage{graphicx,mathrsfs}
 \usepackage{appendix}

\usepackage[colorlinks,linkcolor=blue, citecolor=blue]{hyperref}
\numberwithin{equation}{section}

\usepackage[numbers,sort&compress]{natbib}

\usepackage{esint}
\usepackage{graphicx}
\usepackage[dvipsnames]{xcolor}

\newcommand\e\varepsilon
\newtheorem{theorem}{Theorem}[section]

\newtheorem{proposition}{Proposition}[section]
\newtheorem{definition}{Definition}[section]

\newtheorem{remark}{Remark}[section]

\begin{document}

\title[The kernel space of linear operator on a class of Grushin equation]
{The kernel space of linear operator on a class of Grushin equation}

\author[Y. Wei and X. Zhou]{Yawei Wei and Xiaodong Zhou}

\address[Yawei Wei]{School of Mathematical Sciences and LPMC, Nankai University, Tianjin 300071, China}
\email{weiyawei@nankai.edu.cn}

\address[Xiaodong Zhou]{School of Mathematical Sciences, Nankai University, Tianjin 300071, China}
\email{1120210030@mail.nankai.edu.cn}

\thanks{Acknowledgements: This work is supported by the NSFC under the grands 12271269 and the Fundamental Research Funds for the Central Universities.}




\begin{abstract}%
	In this paper, we concern the kernel of linear operator for a class of Grushin equation. First, we study the kernel space of linear operator for a general Grushin equation. Then, we provide an exact expression for the kernel space of linear operator for a special Grushin equation. Finally, we prove the linear operator related to the singularly perturbed Grushin equation is invertible when restricted to the complement of its approximate kernel space.
\end{abstract}

\maketitle

\section{Introduction}

\setcounter{equation}{0}

In this paper, we consider the following problem
\begin{equation}\label{t1}
  -\Delta_\gamma u=f(u),\;\;u>0,\;\;~\mbox{in}~\mathbb{R}^{N+l},
\end{equation}
where $\Delta_\gamma u$ is well-known \textit{Grushin operator} given by
\begin{equation}\label{i2}
  \Delta_\gamma u(z)=\Delta_xu(z)+|x|^{2\gamma}\Delta_yu(z).
\end{equation}
$\Delta_x$ and $\Delta_y$ are the Laplace operators in the variable $x$ and $y$ respectively, with $z=(x,y)\in\mathbb{R}^N\times\mathbb{R}^l=\mathbb{R}^{N+l}$ and $N>1$, $l>1$, $N+l\geq3$. Here, $\gamma\geq0$ is a real number and
\begin{equation}\label{yw1}
  N_\gamma=N+(1+\gamma)l
\end{equation}
is the appropriate homogeneous dimension. The precise conditions of the nonlinearity $f\in C^1(\mathbb{R})$ will be given in the following.

It's worth noting that $\Delta_\gamma$ is elliptic for $x\neq0$ and degenerates on $\{0\}\times\mathbb{R}^l$. When $\gamma>0$ is an integer, the vector fields
\begin{equation}\label{zzxxyy}
  X_1=\frac{\partial}{\partial x_1}, \cdots, X_N=\frac{\partial}{\partial x_N}, Y_1=|x|^\gamma\frac{\partial}{\partial y_1}, \cdots, Y_l=|x|^\gamma\frac{\partial}{\partial y_l}
\end{equation}
satisfy the H\"ormander condition and then $\Delta_\gamma$ is hypoelliptic. More about H\"ormander condition and hypoellipticity could be found in \cite{mander}. Geometrically, $\Delta_\gamma$ comes from a sub-Laplace operator on a nilpotent Lie group of step $\gamma+1$ by a submersion. Specific explanation of geometric framework can be consulted in \cite{Bauer1} by Bauer et al. In addition, we referred to \cite{Tri4,Tri5}.

In the past few decades, many scholars have concerned with the kernel of linear operator for the following Laplace equation
\begin{equation}\label{zzz11}
\begin{cases}
-\Delta u=f(u),\;\;u>0,&~ \mbox{in} ~\mathbb{R}^{N},\\[2mm]
u\in H^1(\mathbb{R}^N),
\end{cases}
\end{equation}
where $f:\mathbb{R}\rightarrow\mathbb{R}$ is a general nonlinearity and $H^1(\mathbb{R}^N)$ is a standard Sobolev space. In this paper, we consider the kernel of linear operator for the equation \eqref{t1}, which is similar to the case of \eqref{zzz11}.

Special case of \eqref{zzz11} is the following problem
\begin{equation}\label{zzz1}
\begin{cases}
-\Delta u+ u=u^p,\;\;u>0,&~ \mbox{in} ~\mathbb{R}^{N},\\[2mm]
u\in H^1(\mathbb{R}^N),
\end{cases}
\end{equation}
where $1<p<2^*-1=\frac{N+2}{N-2}$ (see \cite{Ni}, for example). Let $U$ be a radially symmetric solution of \eqref{zzz1}, then the kernel of $-\Delta u+u-pU^{p-1}u$, $u\in H^1(\mathbb{R}^N)$, is given by
\begin{equation}\label{zzz2}
  span\left\{\frac{\partial U}{\partial x_1},\cdots,\frac{\partial U}{\partial x_N}\right\}.
\end{equation}
For the proof of the above statement, the readers can refer to \cite{xtm}. In this paper, we give an exact expression for the kernel space of linear operator for the following problem
\begin{equation}\label{hx1}
\begin{cases}
 -\Delta_\gamma u+u=u^p,\;\;u>0,\;\;~\mbox{in}~\mathbb{R}^{N+l},\\[2mm]
u\in H^{1,2}_\gamma(\mathbb{R}^{N+l}),
\end{cases}
\end{equation}
where $p>1$ and $H^{1,2}_\gamma(\mathbb{R}^{N+l})$ is a weighted Sobolev space will be defined later.

Owing to the result of \eqref{zzz2}, many scholars are able to construct solutions of equations by using finite-dimensional reduction methods. Authors in \cite{bib7} construct the multi-peak solutions for the Chern-Simons-Schr\"odinger system. The crucial step to apply implicit function theorem is to prove linear operator $L_\varepsilon$ restricted to $E_{\varepsilon,Y}$ is invertible, where $L_\varepsilon$ is the linear operator of Chern-Simons-Schr\"odinger system
\begin{equation}\label{i8}
\begin{cases}
-\varepsilon^2\Delta u+V(x)u+(A_0+A_1^2+A_2^2)u=|u|^{p-2}u,~\;\;x\in\mathbb{R}^2,\\[2mm]
\partial_1A_0=A_2u^2,~\;\;\partial_2A_0=-A_1u^2 ,\\[2mm]
\partial_1A_2-\partial_2A_1=-\frac{1}{2}|u|^2,~\;\;\partial_1A_1+\partial_2A_2=0,
\end{cases}
\end{equation}
and
\begin{equation}\label{i9}
  E_{\varepsilon,Y}:=\{\varphi\in H_\varepsilon:\langle\frac{\partial U^i_{\varepsilon,y_i}}{\partial y_l^i},\varphi\rangle_\varepsilon=0, i=1,\cdots,k, l=1,2\}
\end{equation}
is the complementary space of the approximate kernel of $L_\varepsilon$. In \eqref{i9},
\begin{equation*}
  H_\varepsilon=\left\{u\in H^1(\mathbb{R}^2): \|u\|_\varepsilon^2=\int_{\mathbb{R}^2}\varepsilon^2|\nabla u|^2+V(x)|u|^2dx<\infty\right\}.
\end{equation*}
Thus it is vital to find the kernel or the approximate kernel of linear operator for the corresponding equation. The approximate kernel of linear operator for \eqref{i8} is $K_{\varepsilon,Y}=\textmd{span}\{\frac{\partial U^i_{\varepsilon,y_i}}{\partial y_l^i}, i=1,\cdots,k, l=1,2\}$ exactly. This is the result of Laplace equation, which has been studied wildly such as in \cite{bib6} and \cite{xtm}. In this paper, we consider the following singularly perturbed nonlinear equation as follows
\begin{equation}\label{p10}
\begin{cases}
-\varepsilon^2\Delta_\gamma u+ a(z)u=u^{q-1},\;\;u>0,~ \mbox{in} ~\mathbb{R}^{N+l},\\[2mm]
u\in H^{1,2}_\gamma(\mathbb{R}^{N+l}),
\end{cases}
\end{equation}
where $\varepsilon>0$ is a small parameter, $2<q<2_\gamma^*=\frac{2N_\gamma}{N_\gamma-2}$, and $a(z)\in C^2(\mathbb{R}^{N+l})$ satisfies $0<a_0\leq a(z)\leq a_1$ in $\mathbb{R}^{N+l}$. We obtain that the linear operator related to \eqref{p10} is invertible when restricted to the complement of its approximate kernel space.

Moreover in \cite{bib10}, Liu and his collaborators concern with a critical Grushin-type problem. By applying Lyapunov-Schmidt reduction argument and by attaching appropriate assumptions, they prove that this problem has infinitely many positive multi-bubbling solutions with arbitrarily large energy and cylindrical symmetry. One can refer to \cite{bib9,bib3,bib12}, \cite{bib13} and the references therein. The above works illustrate that the study of the kernel of linear operator is crucial and can promote the development of other issues related to Grushin type equations.

Now we introduce some notations that will be used later on.

\begin{definition}\label{yw2}
When $\gamma\geq0$, we define a weighted Sobolev space
\begin{equation}\label{p3}
  H^{1,2}_\gamma(\mathbb{R}^{N+l})=\left\{u\in L^2(\mathbb{R}^{N+l})\;|\;\frac{\partial u}{\partial x_i}, |x|^\gamma\frac{\partial u}{\partial y_j}\in L^2(\mathbb{R}^{N+l}), i=1,\cdots,N, j=1,\cdots,l\right\}
\end{equation}
as a weighted Sobolev space.
\end{definition}

If $u\in H^{1,2}_\gamma(\mathbb{R}^{N+l})$, we denote the gradient operator $\nabla_\gamma$ as follows:
\begin{equation}\label{p4}
  \nabla_\gamma u=(\nabla_xu, |x|^\gamma\nabla_yu)=(u_{x_1},\cdots,u_{x_N},|x|^\gamma u_{y_1},\cdots,|x|^\gamma u_{y_l}),
\end{equation}
and
\begin{equation}\label{p5}
  |\nabla_\gamma u|^2=|\nabla_xu|^2+|x|^{2\gamma}|\nabla_yu|^2.
\end{equation}
We should note that $H^{1,2}_\gamma(\mathbb{R}^{N+l})$ is a Hilbert space, if we endow with the inner product by
\begin{equation}\label{p6}
  \langle u,v\rangle_\gamma=\int_{\mathbb{R}^{N+l}}\nabla_\gamma u \cdot \nabla_\gamma v+uvdz,
\end{equation}
and the corresponding norm
\begin{equation}\label{p7}
  \|u\|_\gamma=\big(\int_{\mathbb{R}^{N+l}}|\nabla_\gamma u|^2+|u|^2dz\big)^\frac{1}{2}
\end{equation}
is induced by inner product \eqref{p6}.

\begin{remark}\label{yyy1}
Studying problem \eqref{p10}, we define a new norm on $H^{1,2}_\gamma(\mathbb{R}^{N+l})$ as follows,
\begin{equation}\label{02.1}
  \|u\|_\varepsilon:=\left(\int_{\mathbb{R}^{N+l}}\varepsilon^2|\nabla_\gamma u|^2+a(z)u^2dz\right)^{\frac{1}{2}},
\end{equation}
which is induced by the inner product
\begin{equation}\label{02.2}
  \langle u,v\rangle_\varepsilon:=\int_{\mathbb{R}^{N+l}}\varepsilon^2\nabla_\gamma u\cdot\nabla_\gamma v+a(z)uvdz,\;\;\forall u,v\in H^{1,2}_\gamma(\mathbb{R}^{N+l}).
\end{equation}
This norm is not equivalent to \eqref{p7}, we only use this norm when we study \eqref{p10}.
\end{remark}

\begin{definition}\label{zrf1}
Define a new distance in~$\mathbb{R}^{N+l}$:~
\begin{equation}\label{p8}
  d(z,0)=\left(\frac{1}{(1+\gamma)^2}|x|^{2+2\gamma}+|y|^2\right)^{\frac{1}{2+2\gamma}}
\end{equation}
for~$z=(x,y)\in\mathbb{R}^{N+l}$,~and set
\begin{equation}\label{p9}
  \widetilde{B}_r(0):=\{z=(x,y)\in\mathbb{R}^{N+l}|\;d(z,0)<r\}
\end{equation}
be a ball in the sense of this new distance.
\end{definition}

\begin{definition}\label{zrf4}
We say that $u_\varepsilon$ is a single-peak solution of \eqref{p10} concentrated at $z_0=(x_0,y_0)$ if
 $u_\varepsilon$  satisfies  \\
\noindent \textup{(i)} $u_\varepsilon$ has a local maximum point $z_{\varepsilon}=(x_\varepsilon,y_\varepsilon)\in\mathbb{R}^{N+l}$ such that
\begin{equation*}
  z_\varepsilon\to z_0\in \mathbb{R}^{N+l},\;\;~\mbox{as}~\varepsilon\rightarrow0;
\end{equation*}
\noindent \textup{(ii)} For any given $\tau>0$, there exists $R\gg 1$ such that
\begin{equation*}
|u_\varepsilon(z)|\leq \tau~\,\,\mbox{for}\,\,z\in \mathbb{R}^{N+l}\backslash \widetilde{B}_{R\varepsilon}(z_\varepsilon);
\end{equation*}
\noindent \textup{(iii)} There exists $M>0$ such that
\begin{equation*}
u_\varepsilon\leq M.
\end{equation*}
\end{definition}

We state our main results as follows.

Firstly, consider the problem \eqref{t1}. Recall problem \eqref{t1}:
\begin{equation*}
  -\Delta_\gamma u=f(u),\;\;u>0,\;\;~\mbox{in}~\mathbb{R}^{N+l}.
\end{equation*}
We may assume that the solution $u$ of \eqref{t1} is radially symmetric with respect to the variable $y$, and the solution $u$ decays to 0 fast enough when $|y|\rightarrow\infty$. Let $U(x,|y|)=U(x,r)=U$ be a radially symmetric solution of \eqref{t1}. Set
\begin{equation}\label{t2}
  Lu:=-\Delta_\gamma u-f'(U)u=0,\;\;\;\;~\mbox{in}~\mathbb{R}^{N+l}.
\end{equation}
Here $u$ is in $H^{1,2}_\gamma(\mathbb{R}^{N+l})$. Denote $r=|y|$ and let $\omega\in\mathbb{S}^{l-1}$. Then \eqref{t2} transforms into
\begin{equation}\label{t3}
  -\Delta_xu-|x|^{2\gamma}\big(\frac{\partial^2u}{\partial r^2}+\frac{l-1}{r}\frac{\partial u}{\partial r}+\frac{1}{r^2}\Delta_\omega u\big)-f'(U)u=0,\;\;\;\;~\mbox{in}~\mathbb{R}^{N+l}.
\end{equation}

According to \cite{bib6}, let $\mu_0=0$, $\mu_1=\cdots=\mu_l=l-1$, $\mu_j>l-1$, $j\geq l+1$ be all the eigenvalues of the Laplace operator $-\Delta_\omega$ on $\mathbb{S}^{l-1}$. The eigenfunction corresponding to $\mu_0$ is 1, to $\mu_j$ is $\frac{y_j}{|y|}$ for $j=1,\cdots,l$. We denote the eigenfunction corresponding to $\mu_j$ is $\varphi_j$, then we have
\begin{equation}\label{t4}
  -\Delta_\omega\varphi_j=\mu_j\varphi_j,
\end{equation}
i.e.
\begin{equation}\label{t5}
-\Delta_\omega\cdot1=0\cdot1,\;\;-\Delta_\omega\frac{y_j}{|y|}=(l-1)\frac{y_j}{|y|},\;j=1,\cdots,l.
\end{equation}

Let $u(x,r,\omega)$ be a solution of \eqref{t2}, where
\begin{equation}\label{t6}
  u(x,r,\omega)=\sum\limits^\infty_{j=0}v_j(x,r)\varphi_j(\omega).
\end{equation}
Then $v_j\in H^{1,2}_\gamma(\mathbb{R}^{N+1})$ and satisfies

\begin{equation}\label{t7}
  -\Delta_x v_j(x,r)-|x|^{2\gamma}\big(v_j^{rr}(x,r)+\frac{l-1}{r}v_j^r(x,r)-\frac{\mu_j}{r^2}v_j(x,r)\big)-f'(U)v_j(x,r)=0,
\end{equation}
i.e.
\begin{equation}\label{t8}
  -\Delta_\gamma v_j(x,r)-|x|^{2\gamma}\frac{l-1}{r}v_j^r(x,r)+|x|^{2\gamma}\frac{\mu_j}{r^2}v_j(x,r)-f'(U)v_j(x,r)=0,
\end{equation}
where $\Delta_\gamma v_j(x,r)=\Delta_x v_j(x,r)+|x|^{2\gamma}v_j^{rr}(x,r)$, and $v_j^r(x,r)$, $v_j^{rr}(x,r)$ represent the first derivative and the second derivative of $v_j(x,r)$ with respect to $r$. Our main task is to solve $v_j(x,r)$ from \eqref{t8}. The first result can be stated as follows.

\begin{theorem}\label{th4}
Suppose that $U(x,|y|)=U(x,r)=U$ is the radial symmetric solution with regard to variable $y$ of the equation \eqref{t1} in the sense of $H^{1,2}_\gamma(\mathbb{R}^{N+l})$.
Moreover, suppose that \eqref{t6} is the solution of linear operator \eqref{t2}, where $v_j(x,r)$ and $U(x,r)$ satisfy
\begin{equation}\label{p21}
  v_j(x,r)=O(r^{-l+1}),\;\;U^r(x,r)=O(r^{-l+1}),\;\;~\mbox{as}~r\rightarrow+\infty.
\end{equation}
Then
\begin{equation}\label{p22}
v_j(x,r)=\left\{\begin{array}{ll}
c_jU^r(x,r), & j=1,\cdots,l,\\
0, & j>l,
\end{array}\right.
\end{equation}
for some constant $c_j$, and $U^r(x,r)$ represents the derivative of $U$ with respect to $r$.
\end{theorem}

\begin{remark}
Even in the classical nonlinear equation driven by $\Delta$, the term $v_0$ is difficult to calculate. Hence in Theorem \ref{th4}, we can not obtain the precise form of the term $v_0$. But in the nonlinear problem \eqref{hx1}, we obtain the kernel of corresponding linear operator, which is similar to the case of Laplacian.
\end{remark}
Secondly, recall problem \eqref{hx1}:
\begin{equation*}
\begin{cases}
 -\Delta_\gamma u+u=u^p,\;\;u>0,\;\;~\mbox{in}~\mathbb{R}^{N+l},\\[2mm]
u\in H^{1,2}_\gamma(\mathbb{R}^{N+l}).
\end{cases}
\end{equation*}
We consider the corresponding kernel of linear operator for problem \eqref{hx1}
\begin{equation}\label{hx2}
   Lu:=-\Delta_\gamma u+u-pU^{p-1}u,\;\;\;\;~\mbox{in}~\mathbb{R}^{N+l},
\end{equation}
where $u\in H^{1,2}_\gamma(\mathbb{R}^{N+l})$. To find the term $v_0$, we need to study the following equation
\begin{equation}\label{hx3}
  -\Delta_\gamma v_0(x,r)-|x|^{2\gamma}\frac{l-1}{r}v_0^r(x,r)+v_0(x,r)-pU^{p-1}v_0(x,r)=0,\;\;~\mbox{in}~\mathbb{R}^{N+1},
\end{equation}
where $v_0\in H^{1,2}_\gamma(\mathbb{R}^{N+1})$. By using the method similar to the case of $j>l$ in Theorem \ref{th4}, we can prove that \eqref{hx3} has only zero solution. Hence in \eqref{t6}, $v_0=0$. Then we have the following theorem.

\begin{theorem}\label{yw3}
Under the assumption of Theorem \ref{th4}, the kernel of the linear operator \eqref{hx2} in $H^{1,2}_\gamma(\mathbb{R}^{N+l})$ is precisely
\begin{equation}\label{yw4}
  span\left\{\frac{\partial U(z)}{\partial y_1}, \frac{\partial U(z)}{\partial y_2},\cdots,\frac{\partial U(z)}{\partial y_l}\right\}.
\end{equation}
\end{theorem}

Thirdly, we study the problem \eqref{p10}. Recall the problem \eqref{p10}:
\begin{equation*}
\begin{cases}
-\varepsilon^2\Delta_\gamma u+ a(z)u=u^{q-1},\;\;u>0,~ \mbox{in} ~\mathbb{R}^{N+l},\\[2mm]
u\in H^{1,2}_\gamma(\mathbb{R}^{N+l}).
\end{cases}
\end{equation*}
Assume that $z_\varepsilon=(0,y_\varepsilon)$ be the local maximum point of the solution of \eqref{p10} satisfying Definition \ref{zrf4}. Let $U_{\varepsilon,z_\varepsilon}(z)$ be the solution of
\begin{equation}\label{01.11}
\begin{cases}
-\varepsilon^2\Delta_\gamma u+ a(z_\varepsilon)u=u^{q-1},\;\;&u>0,~ \mbox{in} ~\mathbb{R}^{N+l},\\[2mm]
u(z_\varepsilon)=\displaystyle\max_{z\in\mathbb{R}^{N+l}}u(z), ~&u\in H^{1,2}_\gamma(\mathbb{R}^{N+l}),
\end{cases}
\end{equation}
and
\begin{equation}\label{ddd1}
  z_\varepsilon\rightarrow z_0,\;\;\;~\mbox{as}~\varepsilon\rightarrow0,\;\;~\mbox{where}~z_0=(0,y_0).
\end{equation}
We could regard $U_{\varepsilon,z_\varepsilon}(z)$ as an approximate solution to the equation \eqref{p10}.
Let $w(\widetilde{z})$ be the solution of
\begin{equation}\label{3}
  \begin{cases}
-\Delta_\gamma w+ w=w^{q-1},\;\;&w>0,~ \mbox{in} ~\mathbb{R}^{N+l},\\[2mm]
w(0)=\displaystyle\max_{\widetilde{z}\in\mathbb{R}^{N+l}}w(\widetilde{z}), ~&w\in H^{1,2}_\gamma(\mathbb{R}^{N+l}).
\end{cases}
\end{equation}
We know that $w(\widetilde{z})$ is radially symmetric with respect to $y$ variables and exponential decay at infinity from \cite{wyw}.
By using scaling transformation, we can easily find that the relationship between $U_{\varepsilon,z_\varepsilon}(z)$ and $w(\widetilde{z})$ is
\begin{equation}\label{4}
  U_{\varepsilon,z_\varepsilon}(z)=\big(a(z_\varepsilon)\big)^\frac{1}{q-2}w\left(\frac{\sqrt{a(z_\varepsilon)}}{\varepsilon}x,
  \left(\frac{\sqrt{a(z_\varepsilon)}}{\varepsilon}\right)^{1+\gamma}(y-y_\varepsilon)\right).
\end{equation}
The linear operator associated with the equation \eqref{p10} is
\begin{equation}\label{02.6}
  L_\varepsilon u:=-\varepsilon^2\Delta_\gamma u+a(z)u
  -(q-1)U_{\varepsilon,z_\varepsilon}^{q-2}u.
\end{equation}
The linear operator associated with the equation \eqref{01.11} is
\begin{equation}\label{02.66}
  \widetilde{L}_\varepsilon u:=-\varepsilon^2\Delta_\gamma u+a(z_\varepsilon)u
  -(q-1)U_{\varepsilon,z_\varepsilon}^{q-2}u.
\end{equation}
According to Theorem \ref{yw3}, we know that the kernel of linear operator \eqref{02.66} is
\begin{equation}\label{02.16}
  K_\varepsilon:=span\{\frac{\partial U_{\varepsilon,z_\varepsilon}(z)}{\partial y_j}: j=1,\cdots,l\}.
\end{equation}
Therefore, we take $K_\varepsilon$ as the approximate kernel of $L_\varepsilon$. Define
\begin{equation}\label{02.17}
  E_\varepsilon=K_\varepsilon^\bot:=\{\omega\in H^{1,2}_\gamma(\mathbb{R}^{N+l}):\;\langle\omega,\frac{\partial U_{\varepsilon,z_\varepsilon}(z)}{\partial y_j}\rangle_\varepsilon=0\;j=1,\cdots,l\},
\end{equation}
where $\langle u,v\rangle_\varepsilon:=\displaystyle\int_{\mathbb{R}^{N+l}}\varepsilon^2\nabla_\gamma u\cdot\nabla_\gamma v+a(z)uvdz$. Then $H^{1,2}_\gamma(\mathbb{R}^{N+l})=K_\varepsilon\oplus E_\varepsilon$. In the following, we prove the linear operator \eqref{02.6} is invertible when restricted to the space $E_\varepsilon$.

\begin{theorem}\label{th2.11}
Let $L_\varepsilon$ be the linear operator defined in \eqref{02.6}, $K_\varepsilon$ be the kernel space of linear operator \eqref{02.66} defined in \eqref{02.16}, $E_\varepsilon$ be the the complement space of $K_\varepsilon$ defined in \eqref{02.17}, $Q_\varepsilon$ be the projection from $H_\gamma^{1,2}(\mathbb{R}^{N+l})$ to $E_\varepsilon$ as follows
\begin{equation}\label{201}
  Q_\varepsilon u=u-\sum^l_{j=1}b_j\frac{\partial U_{\varepsilon,z_\varepsilon}(z)}{\partial y_j},
\end{equation}
where $U_{\varepsilon,z_\varepsilon}(z)$ is defined in \eqref{4} and $b_j$ satisfies
\begin{equation}\label{202}
  \langle Q_\varepsilon u, \frac{\partial U_{\varepsilon,z_\varepsilon}(z)}{\partial y_k} \rangle_\varepsilon=0,~\mbox{for}~k=1,\cdots,l.
\end{equation}
Then there exist $\varepsilon_0>0$, $\theta_0>0$ and $\rho>0$, such that for any $\varepsilon\in(0,\varepsilon_0]$ and $z_\varepsilon\in\widetilde{B}_{\theta_0}(z_0)$, $Q_\varepsilon L_\varepsilon$ is a bijective mapping in $E_\varepsilon$, moreover
\begin{equation}\label{100}
  \|Q_\varepsilon L_\varepsilon\omega\|_\varepsilon\geq\rho\|\omega\|_\varepsilon,\;\;\;\; \forall\omega\in E_\varepsilon.
\end{equation}
\end{theorem}

The main contributions of this paper are summarized in the following three points. Firstly, the study of the kernel space of the Laplace operator is extensive such as in \cite{bib77,bib88} and the references therein, but there are few results on degenerate operators. In this paper, we study the kernel space of linear operator for nonlinear problem with a kind of classical degenerate operator--Grushin operators. To this end, we overcome the non-symmetry of the Grushin operators in our research process. Secondly, we prove that the kernel space of linear operator associated with nonlinear problem driven by Grushin operator is precisely a linear subspace spanned by finite elements, which corresponds the result of classical nonlinear problem with Laplacian. Thirdly, for the singular perturbed equation with Grushin operator, we obtain that its linear operator is invertible on the orthogonal complement of its approximate kernel space, which is crucial in the process of constructing the solutions.

The paper in the sequel is organized as follows. In Section 2, we study the linear operator related to a general Grushin equation to prove Theorem \ref{th4} by employing the spherical coordinate systems and the separation of variables. In Section 3, we prove Theorem \ref{yw3} to provide an exact expression for the kernel of linear operator for a special Grushin equation. In Section 4, we devote to prove Theorem \ref{th2.11} to obtain the linear operator $L_\varepsilon$ is invertible when restricted to the complement of the kernel of the linear operator $\widetilde{L}_\varepsilon$ by applying blow up method.

\section{Proof of Theorem \ref{th4}}

In this section, we study the kernel of linear operator for the equation \eqref{t1} in Theorem \ref{th4} by using the spherical coordinate systems and the separation of variables.

\noindent\textbf{Proof of Theorem \ref{th4}:}
Since $U(x,|y|)=U(x,r)=U$ is the solution of \eqref{t1}, we have
\begin{equation}\label{t9}
  -\Delta_\gamma U(x,r)-|x|^{2\gamma}\frac{l-1}{r}U^r(x,r)=f(U).
\end{equation}
We differentiate \eqref{t9} with respect to $r$ and get
\begin{equation}\label{t10}
  -\Delta_\gamma U^r(x,r)-|x|^{2\gamma}\frac{l-1}{r}U^{rr}(x,r)+|x|^{2\gamma}\frac{l-1}{r^2}U^r(x,r)=f'(U)U^r(x,r).
\end{equation}
We note that $\mu_j=l-1$ for $j=1,\cdots,l$. By comparing \eqref{t8} and \eqref{t10}, we find that
\begin{equation}\label{t11}
  v_j(x,r)=c_jU^r(x,r),\;\;j=1,\cdots,l,
\end{equation}
for some constant $c_j$, $j=1,\cdots,l$.

To prove $v_j(x,r)=0$ for $j>l$, we compare $v_j(x,r)$ with $U^r(x,r)$. We write \eqref{t7} as
\begin{equation}\label{t12}
  -\Delta_x v_j(x,r)r^{l-1}-|x|^{2\gamma}\big(r^{l-1}v_j^r(x,r)\big)^r+|x|^{2\gamma}\mu_j r^{l-3}v_j(x,r)-r^{l-1}f'(U)v_j(x,r)=0.
\end{equation}
Multiplying \eqref{t12} by $U^r(x,r)$, we obtain
\begin{equation}\label{t13}
\begin{aligned}
  -&\Delta_x v_j(x,r)r^{l-1}U^r(x,r)-|x|^{2\gamma}\big(r^{l-1}v_j^r(x,r)\big)^rU^r(x,r)
  \\+&|x|^{2\gamma}\mu_j r^{l-3}v_j(x,r)U^r(x,r)-r^{l-1}f'(U)v_j(x,r)U^r(x,r)=0.
\end{aligned}
\end{equation}
We write \eqref{t10} as
\begin{equation}\label{t14}
\begin{aligned}
  -&\Delta_x U^r(x,r)r^{l-1}-|x|^{2\gamma}\big(r^{l-1}U^{rr}(x,r)\big)^r
  \\+&|x|^{2\gamma}(l-1)r^{l-3}U^r(x,r)-r^{l-1}f'(U)U^r(x,r)=0.
\end{aligned}
\end{equation}
Multiplying \eqref{t14} by $v_j(x,r)$, we obtain
\begin{equation}\label{t15}
\begin{aligned}
  -&\Delta_x U^r(x,r)r^{l-1}v_j(x,r)-|x|^{2\gamma}\big(r^{l-1}U^{rr}(x,r)\big)^rv_j(x,r)
  \\+&|x|^{2\gamma}(l-1)r^{l-3}U^r(x,r)v_j(x,r)-r^{l-1}f'(U)U^r(x,r)v_j(x,r)=0.
\end{aligned}
\end{equation}
By \eqref{t13}-\eqref{t15}, we obtain
\begin{equation}\label{t16}
\begin{aligned}
  -&\Delta_x v_j(x,r)r^{l-1}U^r(x,r)+\Delta_x U^r(x,r)r^{l-1}v_j(x,r)
  \\-&|x|^{2\gamma}\big(r^{l-1}v_j^r(x,r)\big)^rU^r(x,r)+|x|^{2\gamma}\big(r^{l-1}U^{rr}(x,r)\big)^rv_j(x,r)
  \\+&|x|^{2\gamma}r^{l-3}U^r(x,r)v_j(x,r)\big(\mu_j-(l-1)\big):=I_1+I_2+I_3+I_4+I_5=0.
\end{aligned}
\end{equation}

Suppose that $v_j(x,r)\not\equiv0$ for $j>l$. We could assume that $v_j(x,0)>0$. The situation when $v_j(x,0)<0$ can be considered similarly.

\textbf{Case 1.} Assume that the first zero point of $v_j(x,r)$ is $r_0>0$. Then by integrating each term of \eqref{t16} on $\mathbb{R}^N\times[0,r_0]$, we get some identities as follows.

We first integrate $I_1$ and $I_2$ on $B_r(0)$, where $B_r(0)\subset\mathbb{R}^N$ is a ball in the sense of Euclidean space.
\begin{equation}\label{t17}
\begin{aligned}
     &\int_{B_r(0)}I_1dx\\=&\int_{B_r(0)}-\Delta_x v_j(x,r)r^{l-1}U^r(x,r)dx
  \\=&r^{l-1}\int_{B_r(0)}\nabla_xU^r(x,r)\cdot\nabla_xv_j(x,r)dx
  -r^{l-1}\int_{\partial{B_r(0)}}U^r(x,r)\frac{\partial v_j(x,r)}{\partial\nu_x}dS.
\end{aligned}
\end{equation}
\begin{equation}\label{t18}
\begin{aligned}
     &\int_{B_r(0)}I_2dx\\=&\int_{B_r(0)}\Delta_x U^r(x,r)r^{l-1}v_j(x,r)dx
  \\=&-r^{l-1}\int_{B_r(0)}\nabla_xv_j(x,r)\cdot\nabla_xU^r(x,r)dx
  +r^{l-1}\int_{\partial{B_r(0)}}v_j(x,r)\frac{\partial U^r(x,r)}{\partial\nu_x}dS.
\end{aligned}
\end{equation}
According to our assumption \eqref{p21}, i.e.
\begin{equation*}
  v_j(x,r)=O(r^{-l+1}),\;\;U^r(x,r)=O(r^{-l+1}),\;\;~\mbox{as}~r\rightarrow+\infty,
\end{equation*}
thus as $r\rightarrow+\infty$ we have
\begin{equation}\label{ny1}
  r^{l-1}U^r(x,r)\frac{\partial v_j(x,r)}{\partial\nu_x}\rightarrow0\;\;~\mbox{and}~\;\;  r^{l-1}v_j(x,r)\frac{\partial U^r(x,r)}{\partial\nu_x}
  \rightarrow0.
\end{equation}
When $r\rightarrow+\infty$, we could plus \eqref{t17} and \eqref{t18} and find that:
\begin{equation}\label{t19}
\begin{aligned}
&\int_{\mathbb{R}^N}I_1+I_2\;dx\\=&\int_{\mathbb{R}^N}-\Delta_x v_j(x,r)r^{l-1}U^r(x,r)dx+\int_{\mathbb{R}^N}\Delta_x U^r(x,r)r^{l-1}v_j(x,r)dx
 \\=&r^{l-1}\int_{\mathbb{R}^N}\nabla_xU^r(x,r)\cdot\nabla_xv_j(x,r)dx-r^{l-1}\int_{\mathbb{R}^N}\nabla_xv_j(x,r)\cdot\nabla_xU^r(x,r)dx
 \\=&0.
\end{aligned}
\end{equation}
Then by integrating \eqref{t19} on $[0,r_0]$, the integral is still equal to 0, i.e.
\begin{equation}\label{00000}
\int_0^{r_0}\int_{\mathbb{R}^N}I_1+I_2\;dxdr=0.
\end{equation}

Next, we integrate $I_3$, $I_4$ and $I_5$ on $[0,r_0]$.
\begin{equation}\label{t20}
\begin{aligned}
    \int_0^{r_0}I_3dr=&\int_0^{r_0}-|x|^{2\gamma}\big(r^{l-1}v_j^r(x,r)\big)^rU^r(x,r)dr
 \\=&-|x|^{2\gamma}U^r(x,r)r^{l-1}v_j^r(x,r)\Big|_0^{r_0}+\int_0^{r_0}|x|^{2\gamma}r^{l-1}v_j^r(x,r)U^{rr}(x,r)dr
 \\=&-|x|^{2\gamma}U^r(x,r_0)r_0^{l-1}v_j^r(x,r_0)+\int_0^{r_0}|x|^{2\gamma}r^{l-1}v_j^r(x,r)U^{rr}(x,r)dr
 \\<&\int_0^{r_0}|x|^{2\gamma}r^{l-1}v_j^r(x,r)U^{rr}(x,r)dr,
\end{aligned}
\end{equation}
since $U^r(x,r_0)<0$, $v_j^r(x,r_0)<0$.
\begin{equation}\label{t21}
\begin{aligned}
    \int_0^{r_0}I_4dr=&\int_0^{r_0}|x|^{2\gamma}\big(r^{l-1}U^{rr}(x,r)\big)^rv_j(x,r)dr
 \\=&|x|^{2\gamma}v_j(x,r)r^{l-1}U^{rr}(x,r)\Big|_0^{r_0}-\int_0^{r_0}|x|^{2\gamma}r^{l-1}U^{rr}(x,r)v_j^r(x,r)dr.
 \\=&-\int_0^{r_0}|x|^{2\gamma}r^{l-1}U^{rr}(x,r)v_j^r(x,r)dr
\end{aligned}
\end{equation}
We could plus \eqref{t20} and \eqref{t21} and find that:
\begin{equation}\label{t22}
  \int_0^{r_0}I_3+I_4\;dr<0.
\end{equation}
Then by integrating \eqref{t22} on $\mathbb{R}^N$, the integral is still less than 0, i.e.
\begin{equation}\label{t212}
  \int_{\mathbb{R}^N}\int_0^{r_0}I_3+I_4\;drdx<0.
\end{equation}
\begin{equation}\label{t23}
  \int_0^{r_0}I_5dr=\int_0^{r_0}|x|^{2\gamma}r^{l-3}U^r(x,r)v_j(x,r)\big(\mu_j-(l-1)\big)dr<0,
\end{equation}
since $v_j(x,r)>0$, $\mu_j>l-1$. The integration of \eqref{t23} on $\mathbb{R}^N$ is still less than 0, i.e.
\begin{equation}\label{t213}
  \int_{\mathbb{R}^N}\int_0^{r_0}I_5drdx<0.
\end{equation}

From \eqref{00000}, \eqref{t212} and \eqref{t213}, we obtain the LHS of the integration of \eqref{t16} on $\mathbb{R}^N\times[0,r_0]$ is strictly negative and the RHS of which is equal to $0$. This implies a contradiction. Hence Case 1 is not valid.

\textbf{Case 2.} Suppose that $v_j(x,r)>0$ for all $r>0$. Since $v_j(x,r)\in H^{1,2}_\gamma(\mathbb{R}^{N+1})$, we have
\begin{equation}\label{t24}
  \int_{\mathbb{R}^N}\int_0^{+\infty}|\nabla_xv_j(x,r)|^2+|x|^{2\gamma}\big(v_j^r(x,r)\big)^2+\big(v_j(x,r)\big)^2drdx<+\infty.
\end{equation}
We can find a sequence $\{r_n\}\rightarrow+\infty$ as $n\rightarrow+\infty$, such that
\begin{equation}\label{t25}
  |\nabla_xv_j(x,r_n)|^2+|x|^{2\gamma}\big(v_j^r(x,r_n)\big)^2+\big(v_j(x,r_n)\big)^2\rightarrow0.
\end{equation}
As same as Case 1, we integrate \eqref{t16} on $\mathbb{R}^N\times[0,r_n]$. The sum of the integral of $I_1$ and $I_2$ is as same as before, and we could refer to \eqref{t17}-\eqref{00000} to obtain
\begin{equation}\label{jjb}
   \int_0^{r_n}\int_{\mathbb{R}^N}I_1+I_2\;dxdr=0.
\end{equation}

Now we integrate $I_3$, $I_4$ and $I_5$ on $[0,r_n]$.
\begin{equation}\label{t26}
\begin{aligned}
  \int_0^{r_n}I_3dr=&\int_0^{r_n}-|x|^{2\gamma}\big(r^{l-1}v_j^r(x,r)\big)^rU^r(x,r)dr
\\=&-|x|^{2\gamma}U^r(x,r)r^{l-1}v_j^r(x,r)\Big|_0^{r_n}+\int_0^{r_n}|x|^{2\gamma}r^{l-1}v_j^r(x,r)U^{rr}(x,r)dr
\\=&-|x|^{2\gamma}U^r(x,r_n)r_n^{l-1}v_j^r(x,r_n)+\int_0^{r_n}|x|^{2\gamma}r^{l-1}v_j^r(x,r)U^{rr}(x,r)dr
\\=&\int_0^{r_n}|x|^{2\gamma}r^{l-1}v_j^r(x,r)U^{rr}(x,r)dr,
\end{aligned}
\end{equation}
since the condition $U^r(x,r)=O(r^{-l+1}),\;r\rightarrow+\infty$ and \eqref{t25}.
\begin{equation}\label{t27}
\begin{aligned}
    \int_0^{r_n}I_4dr=&\int_0^{r_n}|x|^{2\gamma}\big(r^{l-1}U^{rr}(x,r)\big)^rv_j(x,r)dr
    \\=&|x|^{2\gamma}v_j(x,r)r^{l-1}U^{rr}(x,r)\Big|_0^{r_n}-\int_0^{r_n}|x|^{2\gamma}r^{l-1}U^{rr}(x,r)v_j^r(x,r)dr
    \\=&|x|^{2\gamma}v_j(x,r_n)r_n^{l-1}U^{rr}(x,r_n)-\int_0^{r_n}|x|^{2\gamma}r^{l-1}U^{rr}(x,r)v_j^r(x,r)dr
    \\=&-\int_0^{r_n}|x|^{2\gamma}r^{l-1}U^{rr}(x,r)v_j^r(x,r)dr,
\end{aligned}
\end{equation}
since the condition $U^r(x,r)=O(r^{-l+1}),\;r\rightarrow+\infty$ and \eqref{t25}.
Plus \eqref{t26} and \eqref{t27}, we find that
\begin{equation}\label{t28}
  \int_0^{r_n}I_3+I_4\;dr\rightarrow0,\;\;n\rightarrow\infty.
\end{equation}
The integration of \eqref{t28} on $\mathbb{R}^N$ is still tend to 0, i.e.
\begin{equation}\label{zmn1}
  \int_{\mathbb{R}^N}\int_0^{r_n}I_3+I_4\;drdx\rightarrow0,\;\;n\rightarrow\infty.
\end{equation}
The integration of $I_5$ is as same as \eqref{t23}, then we can similarly obtain that the following estimate
\begin{equation}\label{t29}
  \int_{\mathbb{R}^N}\int_0^{r_n}I_5drdx<0.
\end{equation}
From \eqref{jjb}, \eqref{zmn1} and \eqref{t29}, we obtain the LHS of the integration of \eqref{t16} on $\mathbb{R}^N\times[0,r_n]$ is strictly negative and the RHS of which is equal to $0$. This contradicts with $v_j(x,r)>0$. Hence Case 2 cannot hold. Therefore, the only situation is $v_j(x,r)\equiv0$ for $j>l$. We have proved this theorem.
\begin{flushright}
  $\square$
\end{flushright}

\section{Proof of Theorem \ref{yw3}}

In this section, we prove Theorem \ref{yw3} to provide an exact expression for the kernel of linear operator for the equation \eqref{hx1}, which is based on Theorem \ref{th4}.

\noindent\textbf{Proof of Theorem \ref{yw3}:}
Let $U(x,|y|)=U(x,r)=U$ be the solution of \eqref{hx1}, then we have
\begin{equation}\label{mn1}
  -\Delta_\gamma U(x,r)-|x|^{2\gamma}\frac{l-1}{r}U^r(x,r)+U(x,r)-U^p(x,r)=0.
\end{equation}
We differentiate \eqref{mn1} with respect to $r$ and get
\begin{equation}\label{mn2}
\begin{aligned}
  -&\Delta_\gamma U^r(x,r)-|x|^{2\gamma}\frac{l-1}{r}U^{rr}(x,r)
  +|x|^{2\gamma}\frac{l-1}{r^2}U^r(x,r)
  \\+&U^r(x,r)-pU^{p-1}(x,r)U^r(x,r)=0.
\end{aligned}
\end{equation}
Note that $\mu_0=0$, thus $v_0(x,r)$ satisfies
\begin{equation}\label{mn3}
  -\Delta_\gamma v_0(x,r)-|x|^{2\gamma}\frac{l-1}{r}v_0^r(x,r)+v_0(x,r)-pU^{p-1}(x,r)v_0(x,r)=0.
\end{equation}
Multiplying \eqref{mn2} by $v_0(x,r)$ and subtracting \eqref{mn3} multiplied by $U^r(x,r)$, we obtain
\begin{equation}\label{mn4}
\begin{aligned}
  &-\Delta_\gamma U^r(x,r)v_0(x,r)+\Delta_\gamma v_0(x,r)U^r(x,r)
\\&-|x|^{2\gamma}\frac{l-1}{r}U^{rr}(x,r)v_0(x,r)+|x|^{2\gamma}\frac{l-1}{r}v_0^r(x,r)U^r(x,r)
\\&+|x|^{2\gamma}\frac{l-1}{r^2}U^r(x,r)v_0(x,r):=J_1+J_2+J_3+J_4+J_5=0.
\end{aligned}
\end{equation}

Suppose that $v_0(x,r)\not\equiv0$. We could assume that $v_0(x,0)>0$. The situation when $v_0(x,0)<0$ can be considered similarly.

\textbf{Case 1.} Assume that the first zero point of $v_0(x,r)$ is $r_0>0$. Then by integrating each term of \eqref{mn4} on $\mathbb{R}^N\times[0,r_0]$, we get some identities as follows.

We first integrate $J_1$ and $J_2$ on $\Omega=B_r(0)\times[0,r_0]$, where $B_r(0)\subset\mathbb{R}^N$ is a ball in the sense of Euclidean space. And we denote the unit outward normal of $\partial\Omega$ by $\nu=(\nu_x,\nu_r)$.
\begin{equation}\label{mn5}
\begin{aligned}
  \int_\Omega J_1dxdr
  =&\int_\Omega-\Delta_\gamma U^r(x,r)v_0(x,r)dxdr
  \\=&\int_\Omega\nabla_\gamma v_0(x,r)\cdot\nabla_\gamma U^r(x,r)dxdr
  \\-&\int_{\partial\Omega}v_0(x,r)\left(\frac{\partial U^r(x,r)}{\partial\nu_x}+|x|^{2\gamma}\frac{\partial U^r(x,r)}{\partial\nu_r}\right)dS.
\end{aligned}
\end{equation}
\begin{equation}\label{mn6}
\begin{aligned}
  \int_\Omega J_2dxdr
  =&\int_\Omega\Delta_\gamma v_0(x,r)U^r(x,r)dxdr
  \\=&-\int_\Omega\nabla_\gamma U^r(x,r)\cdot\nabla_\gamma v_0(x,r)dxdr
  \\+&\int_{\partial\Omega}U^r(x,r)\left(\frac{\partial v_0(x,r)}{\partial\nu_x}+|x|^{2\gamma}\frac{\partial v_0(x,r)}{\partial\nu_r}\right)dS.
\end{aligned}
\end{equation}
According to condition \eqref{p21}, i.e.
\begin{equation*}
  v_j(x,r)=O(r^{-l+1}),\;\;U^r(x,r)=O(r^{-l+1}),\;\;~\mbox{as}~r\rightarrow+\infty,
\end{equation*}
thus as $r\rightarrow+\infty$ we have
\begin{equation}\label{mn7}
  \int_{\partial\Omega}v_0(x,r)\left(\frac{\partial U^r(x,r)}{\partial\nu_x}+|x|^{2\gamma}\frac{\partial U^r(x,r)}{\partial\nu_r}\right)dS\rightarrow0,
\end{equation}
and
\begin{equation}\label{mn8}
  \int_{\partial\Omega}U^r(x,r)\left(\frac{\partial v_0(x,r)}{\partial\nu_x}+|x|^{2\gamma}\frac{\partial v_0(x,r)}{\partial\nu_r}\right)dS\rightarrow0.
\end{equation}
When $r\rightarrow+\infty$, we could plus \eqref{mn5} and \eqref{mn6} and find that:
\begin{equation}\label{mn9}
\int_{\mathbb{R}^N\times[0,r_0]}J_1+J_2\;dxdr=0.
\end{equation}

Next, we integrate $J_3$, $J_4$ and $J_5$ on $[0,r_0]$.
\begin{equation}\label{mn10}
\begin{aligned}
  \int_0^{r_0}J_3dr
  &=\int_0^{r_0}-|x|^{2\gamma}\frac{l-1}{r}U^{rr}(x,r)v_0(x,r)dr
\\&=-|x|^{2\gamma}\frac{l-1}{r}v_0(x,r)U^r(x,r)\Big|_0^{r_0}+\int_0^{r_0}J_4dr-\int_0^{r_0}J_5dr.
\end{aligned}
\end{equation}
\begin{equation}\label{mn11}
  \int_0^{r_0}J_4dr=\int_0^{r_0}|x|^{2\gamma}\frac{l-1}{r}v_0^r(x,r)U^r(x,r)dr>0,
\end{equation}
since $U^r(x,r_0)<0$, $v_j^r(x,r_0)<0$. Then we have
\begin{equation}\label{mn12}
  \int_0^{r_0}J_3+J_4+J_5\;dr=-|x|^{2\gamma}\frac{l-1}{r}v_0(x,r)U^r(x,r)\Big|_0^{r_0}+2\int_0^{r_0}J_4dr=+\infty.
\end{equation}
By integrating \eqref{mn12} on $\mathbb{R}^N$, the integral is still divergent.

From \eqref{mn4}, \eqref{mn9} and \eqref{mn12}, we obtain the LHS of the integration of \eqref{mn4} on $\mathbb{R}^N\times[0,r_0]$ is divergent and the RHS of which is equal to $0$. This implies a contradiction. Hence Case 1 is not valid.

\textbf{Case 2.} Suppose that $v_0(x,r)>0$ for all $r>0$. Since $v_0(x,r)\in H^{1,2}_\gamma(\mathbb{R}^{N+1})$, we have
\begin{equation}\label{mn13}
  \int_{\mathbb{R}^N}\int_0^{+\infty}|\nabla_xv_0(x,r)|^2+|x|^{2\gamma}\big(v_0^r(x,r)\big)^2+\big(v_0(x,r)\big)^2drdx<+\infty.
\end{equation}
We can find a sequence $\{r_n\}\rightarrow+\infty$ as $n\rightarrow+\infty$, such that
\begin{equation}\label{mn14}
  |\nabla_xv_0(x,r_n)|^2+|x|^{2\gamma}\big(v_0^r(x,r_n)\big)^2+\big(v_0(x,r_n)\big)^2\rightarrow0.
\end{equation}
As same as Case 1, we integrate \eqref{mn4} on $\mathbb{R}^N\times[0,r_n]$. The sum of the integral of $J_1$ and $J_2$ is as same as the process of \eqref{mn5}--\eqref{mn9}, we obtain
\begin{equation}\label{mn15}
  \int_{\mathbb{R}^N\times[0,r_n]}J_1+J_2\;dxdr=0.
\end{equation}
Next, we integrate $J_3$, $J_4$ and $J_5$ on $[0,r_n]$.
\begin{equation}\label{mn16}
\begin{aligned}
  \int_0^{r_n}J_4dr&=\int_0^{r_n}|x|^{2\gamma}\frac{l-1}{r}v_0^r(x,r)U^r(x,r)dr
  \\&=|x|^{2\gamma}\frac{l-1}{r}U^r(x,r)v_0(x,r)\Big|_0^{r_n}+\int_0^{r_n}J_3dr+\int_0^{r_n}J_5dr.
\end{aligned}
\end{equation}
\begin{equation}\label{mn17}
  \int_0^{r_n}J_3dr=\int_0^{r_n}-|x|^{2\gamma}\frac{l-1}{r}U^{rr}(x,r)v_0(x,r)dr<0,
\end{equation}
since $U^{rr}(x,r)>0$ according to condition \eqref{p21}.
\begin{equation}\label{mn18}
  \int_0^{r_n}J_5dr=\int_0^{r_n}|x|^{2\gamma}\frac{l-1}{r^2}U^r(x,r)v_0(x,r)dr<0.
\end{equation}
By \eqref{p21} and \eqref{mn14}, as $r_n\rightarrow\infty$ we have
\begin{equation}\label{mn19}
  \int_0^{r_n}J_3+J_4+J_5dr=|x|^{2\gamma}\frac{l-1}{r}U^r(x,r)v_0(x,r)\Big|_0^{r_n}
  +2\int_0^{r_n}J_3dr+2\int_0^{r_n}J_5dr=-\infty.
\end{equation}
By integrating \eqref{mn19} on $\mathbb{R}^N$, the integral is still divergent.

From \eqref{mn4}, \eqref{mn15} and \eqref{mn19}, we obtain the LHS of the integration of \eqref{mn4} on $\mathbb{R}^N\times[0,r_0]$ is divergent and the RHS of which is equal to $0$. This implies a contradiction. Hence Case 2 is not valid. Therefore, the only situation is $v_0(x,r)\equiv0$. We complete the proof.
\begin{flushright}
$\square$
\end{flushright}

\section{Proof of Theorem \ref{th2.11}}

In this section, we prove the linear operator $L_\varepsilon$ is invertible when restricted to the space $E_\varepsilon$ in Theorem \ref{th2.11} by applying blow up method.

Before that, we study the projection on $E_\varepsilon$ for any function $u\in H_\gamma^{1,2}(\mathbb{R}^{N+l})$.
We define the projection $Q_\varepsilon$ from $H_\gamma^{1,2}(\mathbb{R}^{N+l})$ to $E_\varepsilon$ as follows
\begin{equation*}
  Q_\varepsilon u=u-\sum^l_{j=1}b_j\frac{\partial U_{\varepsilon,z_\varepsilon}(z)}{\partial y_j},
\end{equation*}
where $U_{\varepsilon,z_\varepsilon}(z)$ is defined in \eqref{4} and $b_j$ satisfies
\begin{equation*}
  \langle Q_\varepsilon u, \frac{\partial U_{\varepsilon,z_\varepsilon}(z)}{\partial y_k} \rangle_\varepsilon=0,~\mbox{for}~k=1,\cdots,l,
\end{equation*}
i.e.
\begin{equation}\label{203}
  \sum^l_{j=1}\left\langle \frac{\partial U_{\varepsilon,z_\varepsilon}(z)}{\partial y_j},\frac{\partial U_{\varepsilon,z_\varepsilon}(z)}{\partial y_k}\right\rangle_\varepsilon b_j
  =\langle u,\frac{\partial U_{\varepsilon,z_\varepsilon}(z)}{\partial y_k}\rangle_\varepsilon,~\mbox{for}~k=1,\cdots,l.
\end{equation}

Next proposition tells us we can solve $b=(b_1,b_2,\cdots b_l)$ from \eqref{203}.

\begin{proposition}\label{lem00}
The algebraic equation \eqref{203} with respect to the variable $b_j$ is solvable, and we have the following estimate
\begin{equation}\label{204}
  |b_j|\leq C\varepsilon^{-\frac{(1+\gamma) (l-2)}{2}}\|u\|_\varepsilon.
\end{equation}
\end{proposition}

\begin{proof}
Since $U_{\varepsilon,z_\varepsilon}(z)$ satisfies
\begin{equation}\label{205}
  -\varepsilon^2\Delta_\gamma U_{\varepsilon,z_\varepsilon}(z)+a(z_\varepsilon)U_{\varepsilon,z_\varepsilon}(z)=U_{\varepsilon,z_\varepsilon}^{q-1}(z).
\end{equation}
Then differentiate with respect to $y_j$ for \eqref{205}, and it holds
\begin{equation}\label{206}
  -\varepsilon^2\Delta_\gamma \frac{\partial U_{\varepsilon,z_\varepsilon}(z)}{\partial y_j}+a(z_\varepsilon)\frac{\partial U_{\varepsilon,z_\varepsilon}(z)}{\partial y_j}=(q-1)U_{\varepsilon,z_\varepsilon}^{q-2}(z)\frac{\partial U_{\varepsilon,z_\varepsilon}(z)}{\partial y_j}.
\end{equation}
For any test function $\psi\in H_\gamma^{1,2}(\mathbb{R}^{N+l})$, we have
\begin{equation}\label{207}
\begin{aligned}
  &\int_{\mathbb{R}^{N+l}}\varepsilon^2\nabla_\gamma\frac{\partial U_{\varepsilon,z_\varepsilon}(z)}{\partial y_j}\cdot\nabla_\gamma\psi+a(z_\varepsilon)\frac{\partial U_{\varepsilon,z_\varepsilon}(z)}{\partial y_j}\psi dz\\=&(q-1)\int_{\mathbb{R}^{N+l}}U_{\varepsilon,z_\varepsilon}^{q-2}(z)\frac{\partial U_{\varepsilon,z_\varepsilon}(z)}{\partial y_j}\psi dz.
\end{aligned}
\end{equation}
We rewrite \eqref{207} and obtain
\begin{equation}\label{208}
\begin{aligned}
  &\int_{\mathbb{R}^{N+l}}\varepsilon^2\nabla_\gamma\frac{\partial U_{\varepsilon,z_\varepsilon}(z)}{\partial y_j}\cdot\nabla_\gamma\psi+a(z)\frac{\partial U_{\varepsilon,z_\varepsilon}(z)}{\partial y_j}\psi dz
  \\=&\int_{\mathbb{R}^{N+l}}\big(a(z)-a(z_\varepsilon)\big)\frac{\partial U_{\varepsilon,z_\varepsilon}(z)}{\partial y_j}\psi dz
  +(q-1)\int_{\mathbb{R}^{N+l}}U_{\varepsilon,z_\varepsilon}^{q-2}(z)\frac{\partial U_{\varepsilon,z_\varepsilon}(z)}{\partial y_j}\psi dz,
\end{aligned}
\end{equation}
i.e.
\begin{equation}\label{209}
\begin{aligned}
  &\langle\frac{\partial U_{\varepsilon,z_\varepsilon}(z)}{\partial y_j},\psi\rangle_\varepsilon
  \\=&\int_{\mathbb{R}^{N+l}}\big(a(z)-a(z_\varepsilon)\big)\frac{\partial U_{\varepsilon,z_\varepsilon}(z)}{\partial y_j}\psi dz
  +(q-1)\int_{\mathbb{R}^{N+l}}U_{\varepsilon,z_\varepsilon}^{q-2}(z)\frac{\partial U_{\varepsilon,z_\varepsilon}(z)}{\partial y_j}\psi dz.
\end{aligned}
\end{equation}
Take $\psi=\frac{\partial U_{\varepsilon,z_\varepsilon}(z)}{\partial y_k}$, then the LHS of \eqref{209} gives that
 \begin{equation}\label{210}
\begin{aligned}
  &\langle\frac{\partial U_{\varepsilon,z_\varepsilon}(z)}{\partial y_j},\frac{\partial U_{\varepsilon,z_\varepsilon}(z)}{\partial y_k}\rangle_\varepsilon
  \\=&\int_{\mathbb{R}^{N+l}}\big(a(z)-a(z_\varepsilon)\big)\frac{\partial U_{\varepsilon,z_\varepsilon}(z)}{\partial y_j}\frac{\partial U_{\varepsilon,z_\varepsilon}(z)}{\partial y_k} dz
  \\+&(q-1)\int_{\mathbb{R}^{N+l}}U_{\varepsilon,z_\varepsilon}^{q-2}(z)\frac{\partial U_{\varepsilon,z_\varepsilon}(z)}{\partial y_j}\frac{\partial U_{\varepsilon,z_\varepsilon}(z)}{\partial y_k} dz.
\end{aligned}
\end{equation}
Let
\begin{equation}\label{zmm}
  \widetilde{y}=(\frac{\sqrt{a(0)}}{\varepsilon})^{1+\gamma}y,\;\;\;\;r=|\widetilde{y}|.
\end{equation}
The second term of RHS in \eqref{210} gives that
\begin{equation}\label{211}
\begin{aligned}
&\int_{\mathbb{R}^{N+l}}U_{\varepsilon,z_\varepsilon}^{q-2}(z)\frac{\partial U_{\varepsilon,z_\varepsilon}(z)}{\partial y_j}\frac{\partial U_{\varepsilon,z_\varepsilon}(z)}{\partial y_k} dz
\\=&\int_{\mathbb{R}^{N+l}}Cw^{q-2}(\frac{\sqrt{a(0)}}{\varepsilon}x,r)\left(\frac{\partial w(\frac{\sqrt{a(0)}}{\varepsilon}x,r)}{\partial r}\right)^2
y_jy_k \varepsilon^{(1+\gamma)(l-2)}dxdy
\\=&\varepsilon^{(1+\gamma)(l-2)}\left(\delta_{jk}c_j+o(1)\right),
\end{aligned}
\end{equation}
where $\delta_{jk}=0$ if $j\neq k$ and $\delta_{jj}=1$, $c_j>0$ is a constant, $C=C(a(0))$.\\
Let
\begin{equation}\label{jb}
  \widetilde{x}=\frac{\sqrt{a(z_\varepsilon)}}{\varepsilon}x,\;\;\;\; \widetilde{y}=(\frac{\sqrt{a(z_\varepsilon)}}{\varepsilon})^{1+\gamma}(y-y_\varepsilon),
\end{equation}
then for the first term of RHS in \eqref{210}, we have
\begin{equation}\label{212}
\begin{aligned}
&\int_{\mathbb{R}^{N+l}}\big(a(z)-a(z_\varepsilon)\big)\frac{\partial U_{\varepsilon,z_\varepsilon}(z)}{\partial y_j}\frac{\partial U_{\varepsilon,z_\varepsilon}(z)}{\partial y_k}dz
\\=&\int_{\mathbb{R}^{N+l}}\left(\nabla a(z_\varepsilon)\cdot(z-z_\varepsilon)+O(d^2(z,z_\varepsilon))\right)\frac{\partial U_{\varepsilon,z_\varepsilon}(z)}{\partial y_j}\frac{\partial U_{\varepsilon,z_\varepsilon}(z)}{\partial y_k}dz
\\=&O(\varepsilon^{N_\gamma-1-2\gamma}).
\end{aligned}
\end{equation}
Combining with \eqref{210}, \eqref{211}, and \eqref{212}, finally we get

\begin{equation}\label{213}
  \langle\frac{\partial U_{\varepsilon,z_\varepsilon}(z)}{\partial y_j},\frac{\partial U_{\varepsilon,z_\varepsilon}(z)}{\partial y_k}\rangle_\varepsilon
=\varepsilon^{(1+\gamma)(l-2)}\left(\delta_{jk}c_j+o(1)\right).
\end{equation}
Hence, \eqref{203} is solvable, and we have the following estimate
\begin{equation}\label{214}
\begin{aligned}
  |b_j|\leq& C\varepsilon^{-(1+\gamma)(l-2)}\left|\langle u,\frac{\partial U_{\varepsilon,z_\varepsilon}(z)}{\partial y_k}\rangle_\varepsilon\right|
  \\\leq&C\varepsilon^{-\frac{(1+\gamma)(l-2)}{2}}\|u\|_\varepsilon.
\end{aligned}
\end{equation}
\qed
\end{proof}

\noindent\textbf{Proof of Theorem \ref{th2.11}:}
\noindent\textbf{Step 1: By contradiction to prove \eqref{100}}\\
We will show that there exist $\varepsilon_0>0$, $\theta_0>0$ and $\rho>0$, such that for any $\varepsilon\in(0,\varepsilon_0]$, $z_\varepsilon\in\widetilde{B}_{\theta_0}(z_0)$, and $\omega\in E_\varepsilon$, it satisfies
\begin{equation}\label{cvb}
    \|Q_\varepsilon L_\varepsilon\omega\|_\varepsilon\geq\rho\|\omega\|_\varepsilon.
\end{equation}

We prove this statement by contradiction. Assume that for any $\varepsilon_0>0$, $\theta_0>0$ and $\rho>0$, there exist $\varepsilon_\rho\in(0,\varepsilon_0]$, $z_{\varepsilon_\rho}\in\widetilde{B}_{\theta_0}(z_0)$, and $\omega_\rho\in E_{\varepsilon_\rho}$, such that
\begin{equation}\label{cvb1}
  \|Q_{\varepsilon_\rho} L_{\varepsilon_\rho} \omega_\rho\|_{\varepsilon_\rho}<\rho\|\omega_\rho\|_{\varepsilon_\rho}.
\end{equation}
For convenience, let $\rho=\frac{1}{n}$ and put the above statement in another way, that is to say, suppose that for any given $n>0$, there exist $\varepsilon_n\rightarrow0$, $z_{\varepsilon_n}\rightarrow z_0$, $\omega_n\in E_{\varepsilon_n}$, such that
\begin{equation}\label{101}
  \|Q_{\varepsilon_n} L_{\varepsilon_n} \omega_n\|_{\varepsilon_n}<\frac{1}{n}\|\omega_n\|_{\varepsilon_n}.
\end{equation}
Without loss of generality, we take
\begin{equation}\label{lxp}
  \|\omega_n\|_{\varepsilon_n}^2=\varepsilon_n^{N_\gamma}.
\end{equation}
For any $\phi\in E_{\varepsilon_n}$, by \eqref{02.2} and \eqref{02.6} it holds
\begin{equation}\label{102}
\begin{aligned}
&\int_{\mathbb{R}^{N+l}}\varepsilon_n^2\nabla_\gamma\omega_n \cdot\nabla_\gamma\phi+a(z)\omega_n(z)\phi(z)-(q-1)U_{\varepsilon_n,z_{\varepsilon_n}}^{q-2}(z)\omega_n(z)\phi(z)dz
\\=&\langle L_{\varepsilon_n}\omega_n, \phi\rangle_{\varepsilon_n}=\langle Q_{\varepsilon_n}L_{\varepsilon_n}\omega_n, \phi\rangle_{\varepsilon_n}
\\\overset{\eqref{101}}=&o(1)\|\omega_n\|_{\varepsilon_n}\|\phi\|_{\varepsilon_n}=o(\varepsilon_n^\frac{N_\gamma}{2})\|\phi\|_{\varepsilon_n}.
\end{aligned}
\end{equation}
By taking $\phi=\omega_n$ in \eqref{102}, we can obtain
\begin{equation}\label{103}
\int_{\mathbb{R}^{N+l}}\varepsilon_n^2|\nabla_\gamma\omega_n|^2+a(z)\omega_n^2(z)-(q-1)U_{\varepsilon_n,z_{\varepsilon_n}}^{q-2}(z)\omega_n^2(z)dz=o(\varepsilon_n^{N_\gamma}).
\end{equation}
In \eqref{103}, we can take $R>0$ large enough, such that
\begin{equation}\label{104}
  (q-1)U_{\varepsilon_n,z_{\varepsilon_n}}^{q-2}(z)\leq\frac{1}{2}a(z)\;\;\;\;~\mbox{in}~\mathbb{R}^{N+l}\setminus\widetilde{B}_{R\varepsilon_n}(z_{\varepsilon_n}).
\end{equation}
Then the LHS of \eqref{103} becomes
\begin{equation}\label{105}
\begin{aligned}
&\int_{\mathbb{R}^{N+l}}\varepsilon_n^2|\nabla_\gamma\omega_n|^2+a(z)\omega_n^2(z)-(q-1)U_{\varepsilon_n,z_{\varepsilon_n}}^{q-2}(z)\omega_n^2(z)dz
\\=&\|\omega_n\|_{\varepsilon_n}^2-(q-1)\int_{\widetilde{B}_{R\varepsilon_n}(z_{\varepsilon_n})}U_{\varepsilon_n,z_{\varepsilon_n}}^{q-2}(z)\omega_n^2(z)dz
\\-&\int_{\mathbb{R}^{N+l}\setminus\widetilde{B}_{R\varepsilon_n}(z_{\varepsilon_n})}(q-1)U_{\varepsilon_n,z_{\varepsilon_n}}^{q-2}(z)\omega_n^2(z)dz
\\\overset{\eqref{104}}\geq&\|\omega_n\|_{\varepsilon_n}^2-(q-1)\int_{\widetilde{B}_{R\varepsilon_n}(z_{\varepsilon_n})}U_{\varepsilon_n,z_{\varepsilon_n}}^{q-2}(z)\omega_n^2(z)dz
\\-&\frac{1}{2}\int_{\mathbb{R}^{N+l}\setminus\widetilde{B}_{R\varepsilon_n}(z_{\varepsilon_n})}a(z)\omega_n^2(z)dz
\\\geq&\|\omega_n\|_{\varepsilon_n}^2-(q-1)\int_{\widetilde{B}_{R\varepsilon_n}(z_{\varepsilon_n})}U_{\varepsilon_n,z_{\varepsilon_n}}^{q-2}(z)\omega_n^2(z)dz
-\frac{1}{2}\int_{\mathbb{R}^{N+l}}a(z)\omega_n^2(z)dz
\\\geq&\frac{1}{2}\|\omega_n\|_{\varepsilon_n}^2-(q-1)\int_{\widetilde{B}_{R\varepsilon_n}(z_{\varepsilon_n})}U_{\varepsilon_n,z_{\varepsilon_n}}^{q-2}(z)\omega_n^2(z)dz
\\=&\frac{1}{2}\varepsilon_n^{N_\gamma}-(q-1)\int_{\widetilde{B}_{R\varepsilon_n}(z_{\varepsilon_n})}U_{\varepsilon_n,z_{\varepsilon_n}}^{q-2}(z)\omega_n^2(z)dz,
\end{aligned}
\end{equation}
which, combine with \eqref{103}, we find that
\begin{equation}\label{106}
  \varepsilon_n^{N_\gamma}\leq C\int_{\widetilde{B}_{R\varepsilon_n}(z_{\varepsilon_n})}U_{\varepsilon_n,z_{\varepsilon_n}}^{q-2}(z)\omega_n^2(z)dz
  \leq C'\int_{\widetilde{B}_{R\varepsilon_n}(z_{\varepsilon_n})}\omega_n^2(z)dz.
\end{equation}
In order to obtain contradiction from \eqref{106}, we just need to prove
\begin{equation}\label{107}
  \int_{\widetilde{B}_{R\varepsilon_n}(z_{\varepsilon_n})}\omega_n^2(z)dz=o(\varepsilon_n^{N_\gamma}).
\end{equation}

\noindent\textbf{Step 2: Blow up analysis to prove \eqref{107}}\\
Let
\begin{equation}\label{cvb2}
  \widetilde{x}:=\frac{x}{\varepsilon_n},\quad\widetilde{y}:=\frac{y-y_{\varepsilon_n}}{\varepsilon_n^{1+\gamma}},
\end{equation}
and we define
\begin{equation}\label{108}
  b_n(\widetilde{z})=\omega_n(\varepsilon_n\widetilde{x},\varepsilon_n^{1+\gamma}\widetilde{y}+y_{\varepsilon_n}).
\end{equation}
According to Definition \ref{zrf1}, ball $\widetilde{B}_{R\varepsilon_n}(z_{\varepsilon_n})$ transforms to ball $\widetilde{B}_R(0)$.\\
By \eqref{lxp}, we can calculate that
\begin{equation}\label{109}
  \int_{\mathbb{R}^{N+l}}|\nabla_\gamma b_n(\widetilde{z})|^2+b_n^2(\widetilde{z})d\widetilde{z}=1.
\end{equation}
i.e. $b_n(\widetilde{z})$ is bounded in $H_\gamma^{1,2}(\mathbb{R}^{N+l})$.
Therefore, there exists a subsequence of $\{b_n(\widetilde{z})\}$ in $H_\gamma^{1,2}(\mathbb{R}^{N+l})$, and we still denote by $\{b_n(\widetilde{z})\}$. Also there exists a function $\alpha(\widetilde{z})\in H_\gamma^{1,2}(\mathbb{R}^{N+l})$, such that
\begin{equation}\label{110}
  b_n(\widetilde{z})\rightharpoonup\alpha(\widetilde{z}),\;\;\;\;~\mbox{weakly in}~H_\gamma^{1,2}(\mathbb{R}^{N+l}).
\end{equation}
According to the embedding $H_\gamma^{1,2}(\Omega)\hookrightarrow L^p(\Omega)$ is compact for every $p\in[1,2_\gamma^\ast)$ (see \cite{Berestycki-Lions}).
\begin{equation}\label{111}
  b_n(\widetilde{z})\rightarrow\alpha(\widetilde{z}),\;\;\;\;~\mbox{strongly in}~L_{loc}^2(\mathbb{R}^{N+l}).
\end{equation}
Then the LHS of \eqref{107} becomes
\begin{equation}\label{112}
\begin{aligned}
 &\int_{\widetilde{B}_{R\varepsilon_n}(z_{\varepsilon_n})}\omega_n^2(z)dz
 \\=&\varepsilon_n^{N_\gamma}\int_{\widetilde{B}_R(0)}b_n^2(\widetilde{z})d\widetilde{z}
 \\\overset{\eqref{111}}=&\varepsilon_n^{N_\gamma}\left(\int_{\widetilde{B}_R(0)}\alpha^2(\widetilde{z})d\widetilde{z}+o(1)\right).
\end{aligned}
\end{equation}
To prove \eqref{107}, we claim that $\alpha(\widetilde{z})=0$. For this purpose, we will find the equation that $\alpha(\widetilde{z})$ satisfies.\\

\noindent\textbf{Step 3: To obtain the equation for $\alpha(\widetilde{z})$}\\
For any $\phi\in E_{\varepsilon_n}$ ,we define
\begin{equation}\label{113}
  \psi(\widetilde{z})=\phi(\varepsilon_n\widetilde{x},\varepsilon_n^{1+\gamma}\widetilde{y}+y_{\varepsilon_n}).
\end{equation}
From \eqref{102} and \eqref{cvb2}, we obtain
\begin{equation}\label{114}
\begin{aligned}
  &\int_{\mathbb{R}^{N+l}}\nabla_\gamma b_n(\widetilde{z})\cdot\nabla_\gamma \psi(\widetilde{z})+a(\varepsilon_n\widetilde{x},\varepsilon_n^{1+\gamma}\widetilde{y}+y_{\varepsilon_n})b_n(\widetilde{z})
  \psi(\widetilde{z})
  \\&-(q-1)\overline{U}^{q-2}(\widetilde{z})b_n(\widetilde{z})\psi(\widetilde{z})d\widetilde{z}=o(1)\|\psi(\widetilde{z})\|_\gamma.
\end{aligned}
\end{equation}
where
\begin{equation}\label{116}
  \overline{U}(\widetilde{z})=\big(a(z_{\varepsilon_n})\big)^{\frac{1}{q-2}}w(\sqrt{a(z_{\varepsilon_n})}\widetilde{x},(\sqrt{a(z_{\varepsilon_n})})^{1+\gamma}\widetilde{y}).
\end{equation}
Let $n\rightarrow\infty$, for any $\phi\in E_{\varepsilon_n}$ we have that $\alpha(\widetilde{z})$ satisfies
\begin{equation}\label{117}
  -\Delta_\gamma \alpha(\widetilde{z})+a(z_0)\alpha(\widetilde{z})-(q-1)\overline{U}_0^{q-2}(\widetilde{z})
  \alpha(\widetilde{z})=0,
\end{equation}
where
\begin{equation}\label{118}
  \overline{U}_0(\widetilde{z})=\big(a(z_0)\big)^{\frac{1}{q-2}}w(\sqrt{a(z_0)}\widetilde{x},(\sqrt{a(z_0)})^{1+\gamma}\widetilde{y}).
\end{equation}

Next for any $\phi\in H_\gamma^{1,2}(\mathbb{R}^{N+l})$, we prove $\alpha(\widetilde{z})$ satisfies \eqref{117}. We take
\begin{equation}\label{119}
  Q_{\varepsilon_n}\phi=\phi-\sum^l_{j=1}d_j\frac{\partial U_{\varepsilon_n,z_{\varepsilon_n}}(z)}{\partial y_j}\in E_{\varepsilon_n}.
\end{equation}
Then $d=(d_1, d_2, \ldots d_l)$ satisfies
\begin{equation}\label{120}
  \sum^l_{j=1}\langle\frac{\partial U_{\varepsilon_n,z_{\varepsilon_n}}(z)}{\partial y_j},\frac{\partial U_{\varepsilon_n,z_{\varepsilon_n}}(z)}{\partial y_k}\rangle_{\varepsilon_n}d_j
  =\langle\phi,\frac{\partial U_{\varepsilon_n,z_{\varepsilon_n}}(z)}{\partial y_k}\rangle_{\varepsilon_n},\;\;k=1,\cdots,l.
\end{equation}
According to Proposition \ref{lem00}, \eqref{120} is solvable.
By putting $Q_{\varepsilon_n}\phi$ into \eqref{102}, we have
\begin{equation}\label{121}
\begin{aligned}
&\langle L_{\varepsilon_n}\omega_n, Q_{\varepsilon_n}\phi\rangle_{\varepsilon_n}
\\=&\int_{\mathbb{R}^{N+l}}\varepsilon_n^2\nabla_\gamma\omega_n\cdot\nabla_\gamma\big(\phi-\sum^l_{j=1}d_j\frac{\partial U_{\varepsilon_n,z_{\varepsilon_n}}(z)}{\partial y_j}\big)+a(z)\omega_n(z)\big(\phi-\sum^l_{j=1}d_j\frac{\partial U_{\varepsilon_n,z_{\varepsilon_n}}(z)}{\partial y_j}\big)
\\&-(q-1)U_{\varepsilon_n,z_{\varepsilon_n}}^{q-2}(z)\omega_n(z)\big(\phi-\sum^l_{j=1}d_j\frac{\partial U_{\varepsilon_n,z_{\varepsilon_n}}(z)}{\partial y_j}\big)dz
\\=&\langle L_{\varepsilon_n}\omega_n, \phi\rangle_{\varepsilon_n}-\sum^l_{j=1}d_j r_{n,j},
\end{aligned}
\end{equation}
where
\begin{equation}\label{122}
  r_{n,j}=\langle L_{\varepsilon_n}\omega_n, \frac{\partial U_{\varepsilon_n,z_{\varepsilon_n}}(z)}{\partial y_j}\rangle_{\varepsilon_n}.
\end{equation}
On the other hand,
\begin{equation}\label{123}
\begin{aligned}
  \langle L_{\varepsilon_n}\omega_n, Q_{\varepsilon_n}\phi\rangle_{\varepsilon_n}=&\langle Q_{\varepsilon_n}L_{\varepsilon_n}\omega_n, Q_{\varepsilon_n}\phi\rangle_{\varepsilon_n}
  \\\overset{\eqref{101}}=&o(1)\|\omega_n\|_{\varepsilon_n}\|Q_{\varepsilon_n}\phi\|_{\varepsilon_n}
  \\=&o(\varepsilon_n^\frac{N_\gamma}{2})\|\phi\|_{\varepsilon_n}.
\end{aligned}
\end{equation}
Combining \eqref{121} and \eqref{123}, we obtain
\begin{equation}\label{124}
\begin{aligned}
&\int_{\mathbb{R}^{N+l}}\varepsilon_n^2\nabla_\gamma\omega_n\cdot \nabla_\gamma\phi+a(z)\omega_n(z)\phi(z)-(q-1)U_{\varepsilon_n,z_{\varepsilon_n}}^{q-2}(z)\omega_n(z)\phi(z)dz
\\=&o(\varepsilon_n^\frac{N_\gamma}{2})\|\phi\|_{\varepsilon_n}+\sum^l_{j=1}d_j r_{n,j}.
\end{aligned}
\end{equation}
Now we estimate $r_{n,j}$ by applying blow up analysis, and we claim that
\begin{equation}\label{lxp2}
  r_{n,j}\rightarrow0,\;\;~\mbox{as}~n\rightarrow0.
\end{equation}
In fact, this process of blow up is as same as before. We rewrite \eqref{120} as
\begin{equation}\label{125}
  d_j=\sum^l_{k=1}\alpha_{\varepsilon_n,j,k}\langle\phi,\frac{\partial U_{\varepsilon_n,z_{\varepsilon_n}}(z)}{\partial y_k}\rangle_{\varepsilon_n},
\end{equation}
where $\alpha_{\varepsilon_n,j,k}$ are some constants  only depend on basis of $K_\varepsilon$.
To calculate $r_{n,j}$, we take $\phi=\frac{\partial U_{\varepsilon_n,z_{\varepsilon_n}}(z)}{\partial y_j}$.
By \eqref{213}, \eqref{124} and $\omega_n\in E_{\varepsilon_n}$ , we have
\begin{equation}\label{126}
\begin{aligned}
&\sum^l_{j=1}d_j r_{n,j}
\\\overset{\eqref{124}}=&\langle L_{\varepsilon_n}\omega_n,\frac{\partial U_{\varepsilon_n,z_{\varepsilon_n}}(z)}{\partial y_j}    \rangle_{\varepsilon_n}
+o(\varepsilon_n^\frac{N_\gamma}{2})\|\frac{\partial U_{\varepsilon_n,z_{\varepsilon_n}}(z)}{\partial y_j}\|_{\varepsilon_n}
\\=&\langle \omega_n,\frac{\partial U_{\varepsilon_n,z_{\varepsilon_n}}(z)}{\partial y_j}    \rangle_{\varepsilon_n}
-(q-1)\int_{\mathbb{R}^{N+l}}U_{\varepsilon_n,z_{\varepsilon_n}}^{q-2}(z)\omega_n(z)\frac{\partial U_{\varepsilon_n,z_{\varepsilon_n}}(z)}{\partial y_j}dz
\\&+o(\varepsilon_n^{\frac{N}{2}+(1+\gamma)(l-1)})
\\=&-(q-1)\int_{\mathbb{R}^{N+l}}U_{\varepsilon_n,z_{\varepsilon_n}}^{q-2}(z)\omega_n(z)\frac{\partial U_{\varepsilon_n,z_{\varepsilon_n}}(z)}{\partial y_j}dz
+o(\varepsilon_n^{\frac{N}{2}+(1+\gamma)(l-1)})
\end{aligned}
\end{equation}
Next, we estimate the following term in \eqref{126}:
\begin{equation}\label{127}
  -(q-1)\int_{\mathbb{R}^{N+l}}U_{\varepsilon_n,z_{\varepsilon_n}}^{q-2}(z)\omega_n(z)\frac{\partial U_{\varepsilon_n,z_{\varepsilon_n}}(z)}{\partial y_j}dz.
\end{equation}
Following \eqref{207}, for $\omega_n\in E_{\varepsilon_n}$, we have
\begin{equation}\label{128}
\begin{aligned}
  &\int_{\mathbb{R}^{N+l}}\varepsilon_n^2\nabla_\gamma\frac{\partial U_{\varepsilon_n,z_{\varepsilon_n}}(z)}{\partial y_j}\cdot\nabla_\gamma\omega_n(z)+a(z_{\varepsilon_n})\frac{\partial U_{\varepsilon_n,z_{\varepsilon_n}}(z)}{\partial y_j}\omega_n(z) dz
  \\=&(q-1)\int_{\mathbb{R}^{N+l}}U_{\varepsilon_n,z_{\varepsilon_n}}^{q-2}(z)\frac{\partial U_{\varepsilon_n,z_{\varepsilon_n}}(z)}{\partial y_j}\omega_n(z) dz.
\end{aligned}
\end{equation}
We rewrite \eqref{128} and obtain
\begin{equation}\label{129}
\begin{aligned}
  &\int_{\mathbb{R}^{N+l}}\varepsilon_n^2\nabla_\gamma\frac{\partial U_{\varepsilon_n,z_{\varepsilon_n}}(z)}{\partial y_j}\cdot\nabla_\gamma\omega_n(z)+a(z)\frac{\partial U_{\varepsilon_n,z_{\varepsilon_n}}(z)}{\partial y_j}\omega_n(z) dz
  \\=&\int_{\mathbb{R}^{N+l}}\big(a(z)-a(z_{\varepsilon_n})\big)\frac{\partial U_{\varepsilon_n,z_{\varepsilon_n}}(z)}{\partial y_j}\omega_n(z) dz
  \\+&(q-1)\int_{\mathbb{R}^{N+l}}U_{\varepsilon_n,z_{\varepsilon_n}}^{q-2}(z)\frac{\partial U_{\varepsilon_n,z_{\varepsilon_n}}(z)}{\partial y_j}\omega_n(z) dz.
\end{aligned}
\end{equation}
Notice that
\begin{equation}\label{130}
  \langle\frac{\partial U_{\varepsilon_n,z_{\varepsilon_n}}(z)}{\partial y_j},\omega_n(z)\rangle_{\varepsilon_n}=0.
\end{equation}
According to \eqref{cvb2}, we have
\begin{equation}\label{cvb3}
  z-z_{\varepsilon_n}=(\varepsilon_n\widetilde{x}, \varepsilon_n^{1+\gamma}\widetilde{y}),
\end{equation}
\begin{equation}\label{cvb4}
  d(z,z_{\varepsilon_n})=\varepsilon_nd(\widetilde{z},0), \;\;~\mbox{where}~\widetilde{z}=(\widetilde{x},\widetilde{y}),
\end{equation}
and
\begin{equation}\label{cvb5}
  \frac{\partial U_{\varepsilon_n,z_{\varepsilon_n}}(z)}{\partial y_j}=\frac{\partial U_{\varepsilon_n,z_{\varepsilon_n}}}{\partial \widetilde{y_j}}{\varepsilon_n}^{-(1+\gamma)}.
\end{equation}
Thus by \eqref{129}-\eqref{cvb5} we obtain
\begin{equation}\label{131}
\begin{aligned}
  &-(q-1)\int_{\mathbb{R}^{N+l}}U_{\varepsilon_n,z_{\varepsilon_n}}^{q-2}(z)\omega_n(z)\frac{\partial U_{\varepsilon_n,z_{\varepsilon_n}}(z)}{\partial y_j}dz.
  \\=&\int_{\mathbb{R}^{N+l}}\big(a(z)-a(z_{\varepsilon_n})\big)\frac{\partial U_{\varepsilon_n,z_{\varepsilon_n}}(z)}{\partial y_j}\omega_n(z)dz
  \\=&\int_{\mathbb{R}^{N+l}}\left(\nabla a(z_{\varepsilon_n})\cdot(z-z_{\varepsilon_n})+O(d^2(z,z_{\varepsilon_n}))\right)\frac{\partial U_{\varepsilon_n,z_{\varepsilon_n}}(z)}{\partial y_j}\omega_n(z)dz
  \\=&O(\varepsilon_n^{N_\gamma-\gamma}).
\end{aligned}
\end{equation}
Since for $N>1$,
\begin{equation}\label{cvbn}
  N_\gamma-\gamma>\frac{N}{2}+(1+\gamma)(l-1),
\end{equation}
by \eqref{125}, \eqref{126} and \eqref{131} we get
\begin{equation}\label{132}
\sum^l_{j=1}\sum^l_{k=1}\alpha_{\varepsilon_n,j,k}\langle\frac{\partial U_{\varepsilon_n,z_{\varepsilon_n}}(z)}{\partial y_j},\frac{\partial U_{\varepsilon_n,z_{\varepsilon_n}}(z)}{\partial y_k}\rangle_{\varepsilon_n}r_{n,j}=o(\varepsilon_n^{\frac{N}{2}+(1+\gamma)(l-1)}).
\end{equation}
Combining \eqref{213} with \eqref{132}, we get
\begin{equation}\label{133}
  r_{n,j}=o(\varepsilon^{\frac{N}{2}+1+\gamma})
\end{equation}
Thus we finish the proof of claim.

Therefore, for any $\phi\in H^{1,2}_\gamma(\mathbb{R}^{N+l})$, by \eqref{213} and \eqref{125} we have
\begin{equation}\label{cvbn1}
  \sum^l_{j=1}d_j r_{n,j}=\sum^l_{j=1}\sum^l_{k=1}\alpha_{\varepsilon_n,j,k}\langle\phi,\frac{\partial U_{\varepsilon_n,z_{\varepsilon_n}}(z)}{\partial y_k}\rangle_{\varepsilon_n}r_{n,j}=o(\varepsilon_n^\frac{N_\gamma}{2})\|\phi\|_{\varepsilon_n}.
\end{equation}
Hence for any $\phi\in H^{1,2}_\gamma(\mathbb{R}^{N+l})$, \eqref{124} gives that
\begin{equation}\label{134}
\begin{aligned}
&\int_{\mathbb{R}^{N+l}}\varepsilon_n^2\nabla_\gamma\omega_n \cdot\nabla_\gamma\phi+a(z)\omega_n(z)\phi(z)-(q-1)U_{\varepsilon_n,z_{\varepsilon_n}}^{q-2}(z)\omega_n(z)\phi(z)dz
\\=&o(\varepsilon_n^\frac{N_\gamma}{2})\|\phi\|_{\varepsilon_n}.
\end{aligned}
\end{equation}
According to the definitions \eqref{108} and \eqref{113}, applying blow up analysis, \eqref{134} transforms into \eqref{114}.
Let $n\rightarrow\infty$, therefore $\alpha(\widetilde{z})$ satisfies \eqref{117}.\\

\noindent\textbf{Step 4: To prove $\alpha(\widetilde{z})=0$ }\\

The non-degeneracy of the solution of \eqref{117} tells us
\begin{equation}\label{136}
  \alpha(\widetilde{z})=\sum^l_{j=1}c_j\frac{\partial\overline{U}_0(\widetilde{z})}{\partial\widetilde{y_j}}.
\end{equation}
Then for $\omega_n\in E_{\varepsilon_n}$, we can easily find that
\begin{equation}\label{137}
\begin{aligned}
  0=&\langle\omega_n(z),\frac{\partial U_{\varepsilon_n,z_{\varepsilon_n}}(z)}{\partial y_j}\rangle_{\varepsilon_n}
  \\\overset{\eqref{209}}=&\int_{\mathbb{R}^{N+l}}\big(a(z)-a(z_{\varepsilon_n})\big)\frac{\partial U_{\varepsilon_n,z_{\varepsilon_n}}(z)}{\partial y_j}\omega_n(z) dz
  \\+&(q-1)\int_{\mathbb{R}^{N+l}}U_{\varepsilon_n,z_{\varepsilon_n}}^{q-2}(z)\frac{\partial U_{\varepsilon_n,z_{\varepsilon_n}}(z)}{\partial y_j}\omega_n(z) dz,
\end{aligned}
\end{equation}
in which
\begin{equation}\label{138}
  \int_{\mathbb{R}^{N+l}}\big(a(z)-a(z_{\varepsilon_n})\big)\frac{\partial U_{\varepsilon_n,z_{\varepsilon_n}}(z)}{\partial y_j}\omega_n(z) dz
 \overset{\eqref{131}}=O(\varepsilon_n^{N_\gamma-\gamma}).
\end{equation}
And by \eqref{108} and \eqref{116}, we have
\begin{equation}\label{139}
\begin{aligned}
  &(q-1)\int_{\mathbb{R}^{N+l}}U_{\varepsilon_n,z_{\varepsilon_n}}^{q-2}(z)\frac{\partial U_{\varepsilon_n,z_{\varepsilon_n}}(z)}{\partial y_j}\omega_n(z) dz
  \\=&(q-1)\int_{\mathbb{R}^{N+l}}\overline{U}^{q-2}(\widetilde{z})\frac{\partial\overline{U}(\widetilde{z})}{\partial\widetilde{y_j}}b_n(\widetilde{z})
  \varepsilon_n^{N\gamma-\gamma-1}d\widetilde{z}
\end{aligned}
\end{equation}
Combining with \eqref{137}, \eqref{138}, and \eqref{139}, we obtain
\begin{equation}\label{140}
  \int_{\mathbb{R}^{N+l}}\overline{U}^{q-2}(\widetilde{z})\frac{\partial\overline{U}(\widetilde{z})}{\partial\widetilde{y_j}}b_n(\widetilde{z})
  d\widetilde{z}=O(\varepsilon_n).
\end{equation}
Let $n\rightarrow\infty$,
\begin{equation}\label{141}
    \int_{\mathbb{R}^{N+l}}\overline{U}_0^{q-2}(\widetilde{z})\frac{\partial\overline{U}_0(\widetilde{z})}{\partial\widetilde{y_j}}\alpha(\widetilde{z})
  d\widetilde{z}=0.
\end{equation}
Put \eqref{136} into \eqref{141}, we find that $c_j=0$ in \eqref{136}, $j=1,\cdots,l$. Thus it holds $\alpha(\widetilde{z})=0$.\\
From \noindent\textbf{Step 1} to \noindent\textbf{Step 4}, we obtain \eqref{107} is not valid, which implies that
\eqref{100} holds true.\\

\noindent\textbf{Step 5: To prove $Q_\varepsilon L_\varepsilon$ is a bijective mapping in $E_\varepsilon$}

From \eqref{100} we know $Q_\varepsilon L_\varepsilon$ is a injection in $E_\varepsilon$.

We claim that $Q_\varepsilon L_\varepsilon E_\varepsilon$ is a closed set. In fact, suppose that $v_i=Q_\varepsilon L_\varepsilon u_i\in Q_\varepsilon L_\varepsilon E_\varepsilon$ and $v_i\rightarrow v$, where $u_i\in E_\varepsilon$.
Then
\begin{equation}\label{lxps}
  \|u_i-u_j\|_\varepsilon\leq\frac{1}{\rho}\|v_i-v_j\|_\varepsilon\rightarrow0\;\;~\mbox{as}~i,j\rightarrow\infty.
\end{equation}
Therefore,
\begin{equation}\label{lxpb}
  u_i\rightarrow u\;\;~\mbox{and}~v=Q_\varepsilon L_\varepsilon u\in Q_\varepsilon L_\varepsilon E_\varepsilon.
\end{equation}
Hence, $Q_\varepsilon L_\varepsilon E_\varepsilon$ is a closed set.

To prove $Q_\varepsilon L_\varepsilon$ is a surjection in $E_\varepsilon$, i.e. $Q_\varepsilon L_\varepsilon E_\varepsilon = E_\varepsilon$. We apply proof by contradiction. Suppose that
$Q_\varepsilon L_\varepsilon E_\varepsilon \neq E_\varepsilon$. Then there exists a $v\neq0$ and $v\in E_\varepsilon$. Since $Q_\varepsilon L_\varepsilon E_\varepsilon$ is a closed set, we have the decomposition $v=v_0+w$, where $v_0\in Q_\varepsilon L_\varepsilon E_\varepsilon$ and $w\neq 0$, $w\in(Q_\varepsilon L_\varepsilon E_\varepsilon)^\bot$. Thus for any $u\in E_\varepsilon$,
\begin{equation}\label{142}
  0=\langle Q_\varepsilon L_\varepsilon u,w\rangle_\varepsilon=\langle u,Q_\varepsilon L_\varepsilon w\rangle_\varepsilon,
\end{equation}
which indicates $Q_\varepsilon L_\varepsilon w=0$. Since $Q_\varepsilon L_\varepsilon$ is a injection in $E_\varepsilon$, we get $w=0$, which contradicts that $w\neq 0$. Therefore, $Q_\varepsilon L_\varepsilon$ is a surjection in $E_\varepsilon$.

Since $Q_\varepsilon L_\varepsilon$ is both a injection and a surjection in $E_\varepsilon$, then $Q_\varepsilon L_\varepsilon$ is a bijective mapping in $E_\varepsilon$. This completes the proof.
\begin{flushright}
  $\square$
\end{flushright}



\begin{thebibliography}{99}
\labelsep=1em\relax

\bibitem{Berestycki-Lions}
{\sc C.O. Alves, A.R.F. de Holanda}, {\em A Berestycki-Lions type result for a class of degenerate elliptic problems involving the Grushin operator,} P. Roy. Soc. Edinb. {\bf 153(4)} (2023), 1244--1271.

\bibitem{Bauer1}
{\sc W. Bauer, K. Furutani and C. Iwasaki}, {\em Fundamental solution of a higher step Grushin type operator,} Adv. Math. {\bf 271} (2015) 188--234.

\bibitem{wyw}
{\sc W. Bauer, Y. Wei and X. Zhou}, {\em A priori estimates and moving plane method for a class of Grushin equation,} Preprint.

\bibitem{bib6}
{\sc D.M. Cao, S.J. Peng and S.S. Yan}, {\em Singularity perturbed methods for nonlinear elliptic problems,} Cambridge University Press. 2021.

\bibitem{Ni}
{\sc B. Gidas, W.M. Ni and L. Nirenberg}, {\em Symmetry and related properties via the maximum
principle,} Comm. Math. Phys. {\bf 68} (1979), 209--243.

\bibitem{bib77}
{\sc M. Grossi}, {\em On the number of single-peak solutions of the nonlinear Schr\"odinger equations,}
Ann. Inst. H. Poincar\'e Anal. Non Lineaire. {\bf 19(3)} (2002), 261--280.

\bibitem{bib88}
{\sc Y. Guo, S. Peng and S. Yan}, {\em Local uniqueness and periodicity induced by concentration,} Proc. Lond. Math. Soc.  {\bf 114(6)} (2017), 1005--1043.

\bibitem{bib7}
{\sc Q.Q. Hua, C.H. Wang and J. Yang}, {\em Existence and local uniqueness of multi-peak solutions for the Chern-Simons-Schr\"odinger system,} Arxiv:2210. 17427. 2022.

\bibitem{mander}
{\sc L. H\"{o}rmander}, {\em Hypoelliptic second order differential equations,} Acta Math. {\bf 119} (1967), 141--171.

\bibitem{xtm}
{\sc M.K. Kwong}, {\em Uniqueness of positive solutions of $\Delta u-u+u^p=0$ in $\mathbb{R}^n$,} Arch. Rational Mech. Anal. {\bf 105} (1989), 243--266.

\bibitem{bib10}
{\sc M. Liu, Z.W. Tang and C.H. Wang}, {\em Infinitely many solutions for a critical Grushin-type problem via local Pohozaev identities,} Annali di Mathematica Pure ed Applicata. {\bf 199} (2020), 1737--1762.

\bibitem{bib9}
{\sc P. Luo, S. Peng and C. Wang}, {\em Uniqueness of positive solutions with concentration for the Schr\"odinger-Newton problem,} Calc. Var. Partial Differ. Equ. {\bf 59(2)} (2020), 60.

\bibitem{bib3}
{\sc P.H. Rabinowitz}, {\em Multiple critical points of perturbed symmetric functionals,} Trans. Amer. Math. Soc. {\bf 272(2)} (1982), 753--769.

\bibitem{bib12}
{\sc P. Rabinowitz}, {\em On a class of nonlinear Schr\"odinger equations,} Z. Angew. Math.
Phys. {\bf 43(2)} (1992), 270-291.

\bibitem{Tri4}
{\sc N.M. Tri}, {\em Superlinear equations for degenerate elliptic operators,} Preprint ICTP, World Sci. Publ., Teaneck, N.J. (1995), 1--17.

\bibitem{Tri5}
{\sc N.M. Tri}, {\em On the Grushin equation,} Mat. Zametki. {\bf 63(1)} (1998), 95--105.

\bibitem{bib13}
{\sc X. Wang}, {\em On the concentration of positive bound states of nonlinear Schr\"odinger
equations,} Comm. Math. Phys. {\bf 153} (1993), 229--244.


\end{thebibliography}
\end{document}